\documentclass[a4paper,11pt,reqno]{amsart}

\usepackage{latexsym,dsfont,amssymb,amsmath,amsthm}

\usepackage[all]{xy}

\chardef\bslash=`\\ 

\numberwithin{equation}{section}

\newtheorem{theorem}{Theorem}[section]

\newtheorem{lemma}[theorem]{Lemma}
\newtheorem{proposition}[theorem]{Proposition}

\theoremstyle{remark}
\newtheorem{remark}[theorem]{Remark}

\theoremstyle{definition}

\newcommand\bp{\begin{proof}}
\newcommand\ep{\end{proof}}
\newcommand{\thmref}[1]{Theorem~\ref{#1}}
\newcommand{\secref}[1]{Section~\ref{#1}}
\newcommand{\proref}[1]{Proposition~\ref{#1}}
\newcommand{\lemref}[1]{Lemma~\ref{#1}}

\newcommand\chf{{\mathds 1}}

\newcommand\aaa{\mathfrak a}

\newcommand\pp{\mathfrak p}

\newcommand{\Q}{\mathbb Q}
\newcommand{\R}{\mathbb R}
\newcommand{\C}{\mathbb C}

\newcommand\ak{{\mathbb A}_K}
\newcommand\akf{{\mathbb A}_{K,f}}

\newcommand\al{{\mathbb A}_L}
\newcommand\alf{{\mathbb A}_{L,f}}
\newcommand\OO{{\mathcal O}}
\newcommand\ohs{{\hat{\OO}^*}}
\newcommand\gal{\mathcal G}
\newcommand\kab{K^{ab}}
\newcommand\lab{L^{ab}}
\newcommand\cl{\operatorname{Cl}}
\newcommand\Ind{\operatorname{Ind}}

\newcommand\kcl{\overline{K^*_+}}
\newcommand\ocl{\overline{\OO^*_+}}

\newcommand{\hh}{\mathcal H}

\newcommand\bpmatrix{\begin{pmatrix}}
\newcommand\epmatrix{\end{pmatrix}}

\newcommand{\matr}[2]{\left(\begin{matrix}1&#1 \\
0&#2\end{matrix}\right)}
\newcommand{\diag}[2]{\left(\begin{matrix}#1&0\\
0&#2\end{matrix}\right)}

\newcommand{\Ad}{\operatorname{Ad}}

\newcommand{\supp}{\operatorname{supp}}

\newcommand{\hecke}[2]{\mathcal \hh({#1}, {#2})}
\newcommand{\redheck}[2]{C^*_r({#1}, {#2})}

\newcommand\enu[1]{\smallskip\newline\makebox[6mm][l]{\rm(#1)}}
\newcommand\enuu[1]{\makebox[5mm][l]{\rm(#1)}}

\begin{document}

\title{Bost-Connes systems, Hecke algebras, and induction}
\author[Laca]{Marcelo Laca}
\author[Neshveyev]{Sergey Neshveyev}
\author[Trifkovi\'c]{Mak Trifkovi\'c}

\address{Marcelo Laca, Department of Mathematics and Statistics\\
University of Victoria\\
PO Box 3060\\
Victoria, BC V8W 3R4\\
Canada}
\email{laca@uvic.ca}

\address{Sergey Neshveyev,  Department of Mathematics\\
University of Oslo\\
PO Box 1053 Blindern\\
N-0316 Oslo\\
Norway}
\email{sergeyn@math.uio.no}

\address{Mak Trifkovi\'c, Department of Mathematics and Statistics\\
University of Victoria\\
PO Box 3060\\
Victoria, BC V8W 3R4\\
Canada}
\email{mtrifkov@uvic.ca}

\begin{abstract}
We consider a Hecke algebra naturally associated with the affine group with totally positive multiplicative part
over an algebraic number field $K$
and we show that the C$^*$-algebra of the Bost-Connes system for $K$ can be obtained from our Hecke algebra
by induction, from the group of totally positive principal ideals to the whole group of ideals. Our Hecke algebra is therefore a full corner, corresponding to the narrow Hilbert class field,
in the Bost-Connes C$^*$-algebra of $K$; in particular, the two algebras coincide if and only if $K$ has narrow class number one. Passing the known results for the Bost-Connes system for $K$ to this corner, we obtain a phase transition theorem for our Hecke algebra.

In another application of induction we consider an extension $L/K$ of number fields and we show that the Bost-Connes system for $L$ embeds into the system obtained from the Bost-Connes system for~$K$ by induction from the group of ideals in $K$ to the group of ideals in~$L$. This gives a C$^*$-algebraic correspondence from the Bost-Connes system for $K$ to that for~$L$. Therefore the construction of Bost-Connes systems can be extended to a functor from number fields to C$^*$-dynamical systems with equivariant correspondences as morphisms. We use this correspondence to induce KMS-states and we show that for $\beta>1$ certain extremal KMS$_\beta$-states for $L$ can be obtained, via  induction and rescaling, from KMS$_{[L: K]\beta}$-states for $K$. On the other hand, for $0<\beta\le1$ every KMS$_{[L: K]\beta}$-state for $K$ induces to an infinite weight.
\end{abstract}

\date{October 21, 2010; minor changes April 28, 2013}

\maketitle

\section*{Introduction}

The original  system of Bost and Connes \cite{bos-con} is based on the C$^*$-algebra of  the Hecke pair of orientation-preserving affine groups over the rationals and over the integers. The Bost-Connes Hecke algebra was subsequently shown to be a semigroup crossed product \cite{LRbcalg}, and this realization simplified the analysis of the phase transition and the classification of KMS-states \cite{Ldir,nes}. For general number fields several Hecke algebra constructions have been considered, see e.g.~\cite{HL, ALR,LFr}. In particular, the systems introduced in \cite{LFr} and studied further in \cite{LFr2} exhibit the right phase transition with spontaneous symmetry breaking, but only when the number field has class number one and has no real embeddings. Eventually, however, it was not a Hecke algebra but a restricted groupoid construction modeled on semigroup crossed products that yielded
the generalization of Bost-Connes systems for general number fields
which is now widely regarded as the correct one \cite{cmr1,HP,LLNlat}.
A key step in this construction is the induction from an action of the group
of integral ideles to an action of the Galois group of the maximal abelian extension. In this paper we demonstrate two uses of induction in the study of Bost-Connes type systems for algebraic number fields.

Our first application of induction  appears in \secref{shecke}, where we provide a definitive account of the relation between Bost-Connes systems and ``Hecke systems'' for arbitrary number fields. Specifically, we consider affine groups, over the field and over the algebraic integers, but we restrict the multiplicative subgroup to consist of totally positive elements, that is, to elements that are positive in every real embedding. The resulting inclusion of affine groups is then a Hecke pair and in  \proref{pcross} we show that the corresponding Hecke C$^*$-algebra is a semigroup crossed product which is
a full corner in a group crossed product by the group of totally positive principal ideals. Our main result in this section is \thmref{tiso}, where we show that the Bost-Connes algebra $A_K$ for~$K$ is a corner in the algebra obtained by induction from this crossed product to a crossed product by the full group of fractional ideals over $K$. This realizes our Hecke algebra
as a corner in the Bost-Connes algebra for $K$ and allows us easily to
derive a phase transition with symmetry breaking for our Hecke C$^*$-algebra by importing the known result for Bost-Connes systems from \cite{LLNlat}.

Since our construction restricts multiplication to totally positive elements, the corner is naturally associated to the narrow Hilbert class field $H_+(K)$ of $K$, namely, the maximal abelian extension of~$K$ unramified at every finite prime.   As it turns out, there is a similar crossed product construction for every intermediate field $K \subset L \subset H_+(K)$ between $K$ and its narrow Hilbert class field $H_+(K)$, for which a generalization of our main result holds, see \thmref{tinter}. In particular, when $L=H(K)$ is the Hilbert class field, we get an algebra containing the Hecke algebra of \cite{LFr} as its fixed point subalgebra with respect to the action of a finite subgroup of the Galois group. The rest of \secref{srelation} is devoted to describing relations between phase transitions of the various systems associated to number fields.

Our second application of induction is in \secref{sfunctor}, where we elucidate the functoriality of the construction of a Bost-Connes type system from an algebraic number field. Our main result here is \thmref{tfunctor}, where we show that the construction of Bost-Connes type systems extends to a functor which to an inclusion of number fields $K\hookrightarrow L$ assigns a C$^*$-correspondence which is equivariant with respect to their suitably rescaled natural dynamics. Finally, in \proref{pinduction} we show that  KMS-states of $A_K$ at high inverse temperature pass through the correspondence morphism and, after renormalization and adjusting of the inverse temperature,
they give KMS-states of $A_L$, while other KMS-states, for low inverse temperature, induce to infinite weights and hence do not yield KMS-states of $A_L$.

\bigskip

\section{Algebraic preliminaries} \label{sprelim}

Let $K$ be an algebraic number field with ring of integers
$\OO$.   
For any place $v$ of $K$,
denote by~$K_v$ the completion of~$K$ at $v$.   We indicate that $v$ is finite
(i.e., defined by the valuation at a prime ideal of $\OO$) by writing
$v\nmid\infty$; in that case, let $\OO_v$ be the
closure of $\OO$ in $K_v$.   We similarly put $v|\infty$ when~$v$ is infinite (i.e., defined by an
embedding of $K$ into $\R$ or $\C$), and denote by
$K_\infty=\prod_{v|\infty}K_v$ the completion of~$K$ at all infinite
places. The adele ring $\ak$ is the restricted product, as $v$ ranges
over all places, of the rings
$K_v$, with respect to $\OO_v\subset K_v$ for $v\nmid\infty$. When the product
is taken only over finite places $v$, we get the ring $\akf$ of finite
adeles; we then have $\ak=K_\infty\times\akf$.   The ring of integral
adeles is $\hat\OO=\prod_{v\nmid\infty}\OO_v\subset \akf$.  Let $N_K\colon\akf^*\to(0,+\infty)$ be the absolute norm.

\smallskip

We will need basic facts of class
field theory.  A good general reference is \cite{cf}.
\begin{enumerate}
\item  There exists a continuous surjective homomorphism  $r_K\colon\ak^*\to\gal(K^{ab}/K)$ with kernel $\overline{K^o_\infty K^*}$,
where $K^o_\infty=\prod_{v\text{ real}}\R^*_+\times\prod_{v\text{
 complex}}\C^*$ is the connected component of $K^*_\infty$.
\item \label{efunct} If $\sigma\colon K\hookrightarrow L$ is an embedding of number fields then we have a commutative diagram
$$
\xymatrix{\ak^*\ar[r]^{r_K\ \ \ \ \ }\ar[d]_\sigma & \gal(\kab/K)\ar[d]^{V_{L/\sigma(K)}\circ\Ad\bar\sigma}\\
\al^*\ar[r]_{r_L\ \ \ \ \ } & \gal(\lab/L).}
$$
Here $\bar\sigma\in\gal(\bar\Q/\Q)$ is any extension of $\sigma$, so
that $\Ad\bar\sigma$ defines an isomorphism
$\gal(\kab/K)\to\gal(\sigma(K)^{ab}/\sigma(K))$, and
$V_{L/\sigma(K)}\colon \gal(\sigma(K)^{ab}/\sigma(K))\to\gal(\lab/L)$
is the transfer, or Verlagerung, map.  The definition of this map is
rather involved, but all we will need to know is that it exists
and fits into the above diagram.
\item\label{cfinert}  Let $v$ be a finite place of $K$, and $\bar v$ any extension of
  $v$ to $K^{ab}$.  The inertia group $I_{\bar v/v}$ does not
  depend on the choice of the extension $\bar v$, and satisfies
  $I_{\bar v/v}=r_K(\OO_v^*)$.  Therefore an abelian extension $L/K$ is
  unramified at $v$ if and only if $\OO_v^*$ is in the kernel of the composed map
  $\ak^*\xrightarrow{r_K}\gal(K^{ab}/K)\xrightarrow{restriction}\gal(L/K)$.
\item  The narrow Hilbert class field $H_+(K)$ is the maximal abelian
  extension of $K$ which is unramified at all finite places $v$.  By
  \eqref{cfinert}, we have
  $\gal(K^{ab}/H_+(K))=r_K(\hat\OO^*)\subset\gal(K^{ab}/K)$.
\item  The subfield of $H(K)\subset H_+(K)$ fixed by
  $\gal(\kab/H(K))=r_K(K_\infty^*\hat\OO^*)$ is called the (wide)
  Hilbert class field.  It is characterized by being the maximal abelian
  everywhere unramified extension of $K$, so it is unramified at every finite place and stays real over each real place of~$K$.
\end{enumerate}

It is convenient to remove any reference to infinite places from the above standard statement of class field theory. In order to do this we consider
the multiplicative subgroup  $K_+^*\subset K^*$ of totally positive elements, that is, elements which are positive in every real embedding of $K$. Put also $\OO^\times_+=\OO\cap K^*_+$ and $\OO^*_+=\OO^*\cap K^*_+$. The following isomorphisms are well-known, but for the reader's convenience we still include a proof. The closures considered are in the finite ideles.

\begin{proposition} \label{pgalois}
The restrictions of the Artin map $r_K$ to $\akf^*\supset
K^*\hat\OO^*\supset\hat\OO^*$ give isomorphisms
$$
\akf^*/\overline{K^*_+}\cong\gal(K^{ab}/K),\ \ K^*\hat\OO^*/\kcl\cong\gal(\kab/H(K))\ \ \hbox{and}\ \ \hat\OO^*/\overline{\OO^*_+}\cong\gal(K^{ab}/H_+(K)).
$$
\end{proposition}
Remark: it is stated in \cite[Proposition 4.1]{LFr} that
$\hat\OO^*/\overline{\OO^*}\cong\gal(K^{ab}/H_+(K))$, but the proof
given there works only when all units are totally positive.  The main
results of \cite{LFr} are not affected since they only concern totally
imaginary fields.

\begin{proof}[Proof of Proposition~\ref{pgalois}]
Since $\ak^*=K^o_\infty K^*\akf^*$, the map $\tilde r_K:=r_K|_{\akf^*}\colon\akf^*\to\gal(K^{ab}/K)$ is surjective. Since $K^o_\infty\akf^*$ is open in $\ak^*$, the kernel of the restriction of $r_K$ to $K^o_\infty\akf^*$ is
$$
K^o_\infty\akf^*\cap\overline{K^o_\infty K^*}=\overline{K^o_\infty\akf^*\cap K^o_\infty K^*}=\overline{K^o_\infty K^*_+}.
$$
Hence the kernel of $\tilde r_K$ is the image of $\overline{K^o_\infty K^*_+}$ in $K^o_\infty\akf^*/K^o_\infty=\akf^*$, which is $\overline{K^*_+}\subset \akf^*$. This proves the first isomorphism.

To prove the second isomorphism, observe that $r_K(K_\infty^*)=\tilde r_K(K^*)$. In order to see this denote by~$j$ the embedding of $K^*$ into $\akf^*$. Then $K^*_\infty K^*=K^o_\infty K^* j(K^*)$, whence $r_K(K^*_\infty)=r_K(j(K^*))=\tilde r_K(K^*)$. It follows that $\gal(\kab/H(K))=r_K(K^*_\infty\ohs)=\tilde r_K(K^*\ohs)$. Since $K^*\ohs$ is open in $\akf^*$ and contains~$K^*_+$, which is dense in the kernel of $\tilde r_K$, we get the second isomorphism.

The third isomorphism follows from $\gal(K^{ab}/H_+(K))=\tilde r_K(\hat\OO^*)$ and $\hat\OO^*\cap \overline{K^*_+}=\overline{\OO^*_+}$.
\ep

Let $J_K\cong\akf^*/\hat\OO^*$ be the group of fractional ideals of
$K$ and let $\mathcal P_{K,+}\cong K^*_+/\OO^*_+$ be the subgroup of
principal fractional ideals with a totally positive generator. By the above proposition the preimage of $\gal(K^{ab}/H_+(K))$ in $\akf^*$ is the group $K^*_+\ohs$. Hence
$$
\gal(H_+(K)/K)\cong\akf^*/K^*_+\ohs\cong J_K/\mathcal P_{K,+}.
$$
The last quotient is by definition $\cl_+(K)$, the narrow class group of $K$.

\smallskip

The fundamental construction
underlying this paper is induction. Let $\rho\colon H\to G$ be a homomorphism of groups and $X$ be a set with a left action of $H$.  The formula $h(g,x)=(g\rho(h)^{-1}, hx)$ defines a left action
of $H$ on $G\times X$.  The quotient
$$
G\times_H X:=H\backslash(G\times X)
$$
is called the balanced product associated to the pair $(\rho,X)$, or
the induction of $X$ via $\rho$. There is a natural left
action of $G$ on $G\times_H X$: $g(g',x)=(gg',x)$.  Restricting to
$H$, we get an action of $H$ on $G\times_H X$.  The composition
of the map $X\to G\times X$, $x\mapsto (e,x)$, with the quotient map $G\times X\to G\times_H X$ gives a map $i\colon X\to G\times_H X$. This map is $H$-equivariant in the sense that $i(hx)=\rho(h)i(x)$. It induces a bijection $H\backslash X\cong G\backslash(G\times_H X)$.

Assume now that $G$ and $H$ are discrete groups, $\rho$ is injective, and $X$ is a locally compact space with an action of $H$ by homeomorphisms. In this case $i(X)$ is a clopen subset of $G\times_H X$ and the map $i\colon X\to i(X)$ is a homeomorphism. If the action of $H$ on $X$ is proper, we get a homeomorphism $H\backslash X\cong G\backslash(G\times_H X)$ of locally compact spaces. For general actions there is a version of this homeomorphism for reduced crossed products, thought of as noncommutative quotients. Namely, consider the transformation groupoid $G\times(G\times_H X)$ defined by the action of $G$ on $G\times_H X$. Observe that $g i(X)\cap i(X)\ne\emptyset$ if and only if $g\in\rho(H)$. It follows that the reduction of $G\times(G\times_H X)$ by the open subset $i(X)\subset G\times_H X$ is a groupoid which is isomorphic to the transformation groupoid $H\times X$. Therefore we have the following result.

\begin{proposition}\label{pnoncomind}
Let $\rho\colon H\to G$ be an injective homomorphism of discrete groups,
and let $X$ be a locally compact space with an action of $H$. Then $i(X)$ is a clopen subset of $G\times_H X$, the corresponding projection in the multiplier algebra of $C_0(G\times_H X)\rtimes_r G$ is full, and
$$
C_0(X)\rtimes_r H\cong\chf_{i(X)}(C_0(G\times_H X)\rtimes_r G)\chf_{i(X)}.
$$
\end{proposition}

The same is true for full crossed products. In our applications the group $G$ will be abelian, so that reduced and full crossed products coincide.

\bigskip

\section{From Hecke algebras to Bost-Connes systems}\label{shecke}

For a number field $K$ consider the following inclusion of $ax+b$ groups:
$$
P_\OO^+=\matr{\OO}{\OO^*_+}\subset P_K^+=\matr{K}{K^*_+}.
$$
Recall that a pair of groups $\Gamma\subset G$ is called a Hecke pair
if every double coset can be written as a finite disjoint union of left and right cosets:
$$
\Gamma g\Gamma=\bigsqcup_{i=1}^{L(g)}\Gamma
l_i=\bigsqcup_{j=1}^{R(g)}r_j\Gamma,\ \ g,l_i,r_j\in G.
$$
This happens if and only if the subgroups $\Gamma$ and $g\Gamma
g^{-1}$ are commensurable for every $g\in G$.  In that case, the modular function of the pair is defined by
$$
\Delta(g)=\frac{L(g)}{R(g)}=\frac{[\Gamma\colon \Gamma\cap g\Gamma g^{-1}]}{[g\Gamma g^{-1}\colon \Gamma\cap g\Gamma g^{-1}]}.
$$

\begin{lemma}\label{heckepairs}
The inclusion $P_\OO^+\subset P_K^+$ is a Hecke pair, and for $y\in K$, $x\in K^*_+$ we have
$$
\Delta\matr{y}{x}=N_K(x),
$$
where $N_K\colon\akf^*\to(0,+\infty)$ is the absolute norm.
\end{lemma}

\bp This can be checked by direct computation of double cosets, as in
\cite{LFr}. Alternatively we can embed the pair  $P_\OO^+\subset
P_K^+$ densely into the pair
$$
\bar P_\OO^+=\matr{\hat\OO}{\overline{\OO^*_+}}\subset\bar P_K^+=\matr{\akf}{\overline{K^*_+}}
$$
of subgroups of $\matr{\akf}{\akf^*}$,
and use the theory of topological Hecke pairs as in \cite{tza}.

The group $\bar P^+_K$ is locally compact, and $\bar P^+_\OO$ is a
compact open subgroup, which shows that $(\bar P^+_\OO,\bar P^+_K)$ is
a Hecke pair.  Since $P^+_K$ is dense in $\bar P^+_K$ and
$P^+_\OO=\bar P^+_\OO\cap P^+_K$, it follows that
$(P^+_\OO,P^+_K)$ is also a Hecke
pair. Furthermore, the modular function of $(P^+_\OO,P^+_K)$ is the restriction of the modular function of the locally compact group $\bar P^+_K$ to $P^+_K$.

If $\mu$ and $\nu$ are Haar measures on~$\overline{K^*_+}$ and $\akf$, respectively, then
$$
d\lambda\matr{y}{x}=d\mu(x)d\nu(y)
$$
is a left-invariant Haar measure on $\bar P^+_K$. Since $\nu$ has the
property $\nu(\cdot\,x)=N_K(x)\nu (\cdot)$ for $x\in\akf^*$, we get the
required formula for the modular function of  $(P^+_\OO,P^+_K)$.
\ep

Recall that if $\Gamma\subset G$ is a Hecke pair, then the space $\hecke{G}{\Gamma}$ of finitely supported functions on
$\Gamma\backslash G/\Gamma$ is a $*$-algebra with  product
$$
(f_1*f_2)(g)=\sum_{h\in\Gamma\backslash G}f_1(gh^{-1})f_2(h)
$$
and involution $f^*(g)=\overline{f(g^{-1})}$. Denote by $[g]\in
\hecke{G}{\Gamma}$ the characteristic function of the double coset
$\Gamma g\Gamma$. The Hecke algebra $\hecke{G}{\Gamma}$ is faithfully represented on $\ell^2(\Gamma\backslash G)$ by
$$
(f\xi)(g)=\sum_{h\in\Gamma\backslash G}f(gh^{-1})\xi(h)\ \ \hbox{for}\ \ f\in\hecke{G}{\Gamma}\ \ \hbox{and}\ \ \xi\in\ell^2(\Gamma\backslash G).
$$
Denote by $\redheck{G}{\Gamma}$ the closure of $\hecke{G}{\Gamma}$ in this representation. The C$^*$-algebra $\redheck{G}{\Gamma}$ carries a canonical action of $\R$ defined by $[g]\mapsto\Delta(g)^{-it}[g]$.

\begin{proposition} \label{pcross}
The C$^*$-algebra $\redheck{P^+_K}{P^+_\OO}$ is isomorphic to
$$
\chf_{\hat\OO/\overline{\OO^*_+}}(C_0(\akf/\overline{\OO^*_+})
\rtimes_\alpha(K^*_+/\OO^*_+))\chf_{\hat\OO/\overline{\OO^*_+}},
$$
where the action $\alpha$ of $K^*_+/\OO^*_+$ on $C_0(\akf/\overline{\OO^*_+})$ is defined by $\alpha_x(f)=f(x^{-1}\cdot)$. Furthermore, the isomorphism can be chosen such that the canonical action of $\R$ on $\redheck{P^+_K}{P^+_\OO}$ corresponds to the restriction to the corner of the action $\sigma$ on the crossed product defined by
$$
\sigma_t(fu_x)=N_K(x)^{-it}fu_x\ \ \hbox{for}\ \ f\in C_0(\akf/\overline{\OO^*_+})\ \ \hbox{and}\ \ x\in K^*_+/\OO^*_+,
$$
where the $u_x$ are the canonical unitaries implementing $\alpha$.
\end{proposition}

\bp This is analogous to \cite[Theorem 2.5]{LFr}, so we will be relatively brief. We will use an argument similar to the one in~\cite[Section~3.1]{LLNfcm}.

Consider the groups $\bar P^+_K$ and  $\bar P^+_\OO$ from the previous
lemma. Then $\redheck{\bar P^+_K}{\bar P^+_\OO}$ is canonically
isomorphic to $pC^*_r(\bar P^+_K)p$, where $p=\int_{\bar
  P^+_\OO}u_gd\lambda(g)$ is the projection corresponding to the
compact open  subgroup $\bar P^+_\OO$ (the Haar measure $\lambda$ is assumed
to be normalized so that the measure of $\bar P^+_\OO$ is one). The
projection $p$ is the product of two commuting projections $p_1$ and
$p_2$ corresponding to the subgroups $\matr{\hat\OO}{1}$ and
$\matr{0}{\overline{\OO^*_+}}$, respectively. Since $\bar P^+_K$ is a
semidirect product of $\akf$ and $\overline{K^*_+}$, the C$^*$-algebra
$C^*_r(\bar P^+_K)$ is isomorphic to
$C^*_r(\akf)\rtimes\overline{K^*_+}$. The group $\akf$ is selfdual; we
normalize the isomorphism $\widehat{\akf}\cong\akf$ by requiring that the annihilator of $\hat\OO$ is again $\hat\OO$. Then the image of the projection~$p_1$ under the isomorphism $C^*_r(\akf)\to C_0(\akf)$ is $\chf_{\hat\OO}$. Therefore
\begin{equation} \label{ecorner}
pC^*_r(\bar P^+_K)p\cong\chf_{\hat\OO}p_2(C_0(\akf)\rtimes\overline{K^*_+})p_2\chf_{\hat\OO}.
\end{equation}
The projection $p_2$ corresponding to the subgroup $\overline{\OO^*_+}$ of $\overline{K^*_+}$ commutes with the unitaries $u_x$, $x\in \overline{K^*_+}$, and $p_2C_0(\akf)p_2=C_0(\akf/\overline{\OO^*_+})p_2$. Therefore
$$
p_2(C_0(\akf)\rtimes\overline{K^*_+})p_2=p_2(C_0(\akf/\overline{\OO^*_+})\rtimes\overline{K^*_+})p_2.
$$
We have a surjective $*$-homomorphism $C_0(\akf/\overline{\OO^*_+})\rtimes(K^*_+/\OO^*_+)\to p_2(C_0(\akf/\overline{\OO^*_+})\rtimes\overline{K^*_+})p_2$ which maps $f\in C_0(\akf/\overline{\OO^*_+})$ to $fp_2$ and $u_{\bar x}$, $\bar x\in K^*_+/\OO^*_+$, to $u_xp_2$, where $x\in K^*_+$ is any representative of $\bar x$. To see that this is an isomorphism, assume we have a covariant pair of representations of $C_0(\akf/\overline{\OO^*_+})$ and $K^*_+/\OO^*_+$. Since $K^*_+\cap\overline{\OO^*_+}=\OO^*_+$, the unitary representation of $K^*_+/\OO^*_+$ defines a continuous representation of $\overline{K^*_+}$ with kernel containing $\overline{\OO^*_+}$. Thus we get a covariant pair of representations of $C_0(\akf/\overline{\OO^*_+})$ and $\overline{K^*_+}$ such that the corresponding representation of the crossed product maps $p_2$ into one. Therefore any representation of $C_0(\akf/\overline{\OO^*_+})\rtimes(K^*_+/\OO^*_+)$ factors through $$p_2(C_0(\akf/\overline{\OO^*_+})\rtimes\overline{K^*_+})p_2.$$ Thus
$
p_2(C_0(\akf/\overline{\OO^*_+})\rtimes\overline{K^*_+})p_2\cong C_0(\akf/\overline{\OO^*_+})\rtimes(K^*_+/\OO^*_+),
$
which together with \eqref{ecorner} gives the result.
\ep
The corner $\chf_{\hat\OO/\overline{\OO^*_+}}(C_0(\akf/\overline{\OO^*_+})
\rtimes(K^*_+/\OO^*_+))\chf_{\hat\OO/\overline{\OO^*_+}}$ can also be viewed as the semigroup crossed product
$
C(\hat\OO/\overline{\OO^*_+})\rtimes(\OO^\times_+/\OO^*_+),
$ see \cite[Theorems 2.1 and 2.4]{Ldil}.

\smallskip

As a consequence of the above proposition we see that the group $\hat\OO^*/\overline{\OO^*_+}$ acts on $\redheck{P^+_K}{P^+_\OO}$; the action is however noncanonical, as the isomorphism in the proposition depends on the choice of the isomorphism $\widehat{\akf}\cong\akf$. Recall that by Proposition~\ref{pgalois} we have $\hat\OO^*/\overline{\OO^*_+}\cong\gal(\kab/H_+(K))$.

\smallskip

By Proposition~\ref{pcross} the C$^*$-algebra $\redheck{P^+_K}{P^+_\OO}$ is a full corner in the crossed product algebra defined by the action of $K^*_+/\OO^*_+$ on $\akf/\overline{\OO^*_+}$. We now induce this action via the inclusion $K^*_+/\OO^*_+\cong \mathcal P_{K,+}\hookrightarrow J_K$ of totally
positive principal fractional ideals into all fractional ideals:
$$
X^+_K:=J_K\times_{K^*_+/\OO^*_+}(\akf/\overline{\OO^*_+}).
$$
We equip the crossed product $C_0(X_K^+)\rtimes J_K$
with the dynamics given by
\begin{equation}\label{edynamo}
\sigma^{K,+}_t(fu_g)=N_K(g)^{it}fu_g\ \ \hbox{for}\ \ f\in C_0(X^+_K)\
\ \hbox{and}\ \ g\in J_K,
\end{equation}
where $N_K(g)$ denotes the norm of a fractional ideal $g$. Note that if $g=(x)$ for some $x\in K$, then $N_K(g)=N_K(x)^{-1}$.
Consider also the subset $Y^+_K\subset X^+_K$ defined by
$$
Y^+_K=\{(g,\omega)\in X^+_K\mid g\omega\in \hat\OO/\hat\OO^*\}.
$$
Here we think of $g\in J_K$ as an element of $\akf^*/\hat\OO^*$; then $g\omega$ is a well-defined element of $\akf/\hat\OO^*$. In other words, if we identify $X_K^+$ with a quotient of $\akf^*\times\akf$, then $Y^+_K$ is the image of $\{(g,\omega)\in\akf^*\times\akf\mid g\omega\in\hat\OO\}$. Since $\hat\OO$ is compact and open in $\akf$ and $K^*_+/\OO^*_+$ has finite index in $J_K$, the set $Y_K^+$ is compact and open in $X_K^+$. We put
$$
A_K^+=\chf_{Y^+_K}(C_0(X_K^+)\rtimes J_K)\chf_{Y^+_K}=C(Y_K^+)\rtimes J_K^+,
$$
where $J_K^+\subset J_K$ is the subsemigroup of integral ideals. Since $\sigma^{K,+}$ fixes
$\chf_{Y_K^+}$, it restricts to a dynamics on $A_K^+$, which we continue to
denote by $\sigma^{K,+}$.
Thus, starting from the Hecke algebra $\redheck{P^+_K}{P^+_\OO}$, we have constructed a C$^*$-dynamical system $(A^+_K,\sigma^{K,+})$.

\smallskip

On the other hand, the Bost-Connes system associated with $K$ is defined as follows~\cite{HP,LLNlat}. Consider the balanced
product $X_K=\gal(\kab/K)\times_{\hat\OO^*}\akf$, the induction of the
multiplication action of $\hat\OO^*$ on $\akf$ via the restriction
of the Artin map $\ak^*\to\gal(\kab/K)$ to $\hat\OO^*$.
This space has a natural action of $J_K$, induced from the action of
$\akf^*$ on $\gal(\kab/K)\times\akf$ given by $g(\gamma,x)=(\gamma
r_K(g)^{-1}, gx)$.  Consider the crossed product C$^*$-algebra $C_0(X_K)\rtimes J_K$. Define a dynamics by the same formula as in~\eqref{edynamo}:
$$
\sigma^{K}_t(fu_g)=N_K(g)^{it}fu_g\ \ \hbox{for}\ \ f\in C_0(X_K)\
\ \hbox{and}\ \ g\in J_K.
$$
To define the Bost-Connes system, we pass to the corner
$$
A_K:=\chf_{Y_K}(C_0(X_K)\rtimes J_K)\chf_{Y_K},
$$
corresponding to the compact subspace
$Y_K=\gal(\kab/K)\times_{\hat\OO^*}\hat\OO$.  Since $\sigma^K$ fixes
$\chf_{Y_K}$, it restricts to a dynamics on $A_K$, which we continue to
denote by $\sigma^K$.

\begin{lemma}\label{lxhomeo}
The map $\phi\colon\akf^*\times\akf\to \akf^*\times\akf$, $\phi(x,y)=(x^{-1},xy)
$ induces a $J_K$-equivariant homeomorphism $X_K\cong X_K^+$.
In this homeomorphism $Y_K$ is mapped onto $Y^+_K$, and the set
$$Z_{H_+(K)}=\gal(\kab/H_+(K))\times_{\hat\OO^*}\hat\OO\subset Y_K$$
is mapped  onto $i(\hat\OO/\overline{\OO^*_+})=\{\OO\}\times\hat\OO/\overline{\OO^*_+}$, where $i$ is the canonical embedding $\akf/\overline{\OO^*_+}\hookrightarrow X^+_K$.
\end{lemma}

\begin{proof}
Take two copies of $\akf^*\times\akf$ with the left action of
$\akf^*\times\akf^*$ defined by
$(g,h)(x,y)=(gxh^{-1},hy)$.   Then $\phi((g,h)(x,y))=(h,g)\phi(x,y)$.
Restricting the action to the subgroup $\kcl\times\hat\OO^*$ of $\akf^*\times\akf^*$,
we get a homeomorphism
\begin{equation}
(\akf^*\times\akf)/(\kcl\times\hat \OO^*)\cong (\akf^*\times\akf)/(\hat\OO^*\times\kcl).\label{invol}
\end{equation}
To compute the quotient by $\kcl\times\hat\OO^*$, we can first divide
out by $\kcl$ (which acts only on the first component), and then by
$\hat\OO^*$ (which balances both).  The quotient by
$\hat\OO^*\times\kcl$ is similar.  Therefore the bijection \eqref{invol} gives
the first homeomorphism in
\begin{equation*}\label{doublequots}
(\akf^*/\kcl)\times_{\hat\OO^*}\akf\cong(\akf^*/\hat\OO^*)\times_{\kcl}\akf
\cong(\akf^*/\hat\OO^*)\times_{\kcl/\ocl}\akf/\ocl,
\end{equation*}
the second coming from the fact that $\ocl=\hat\OO^*\cap\kcl$ acts
trivially on $\akf^*/\hat\OO^*$. Since $\kcl/\ocl=K^*_+/\OO^*_+$, we get the desired homeomorphism
$X_K\cong X_K^+$ after identifications $\akf^*/\kcl\cong\gal(\kab/K)$ from Proposition~\ref{pgalois}, and
$\akf^*/\hat\OO^*\cong J_K$.

The map $\phi\colon\akf^*\times\akf\to \akf^*\times\akf$ is $\akf^*$-equivariant with respect to the action $g(x,y)=(xg^{-1},gy)$ on the first space and $g(x,y)=(gx,y)$ on the second. This implies that the homeomorphism $X_K\to X_K^+$ is $J_K$-equivariant.

The subset $Y_K\subset X_K$ is the image of the subset $\akf^*\times\OO\subset\akf^*\times\akf$, while $Y_K^+$ is the image of $\{(x,y)\mid xy\in\OO\}$. We have $\phi(\akf^*\times\OO)=\{(x,y)\mid xy\in\OO\}$, so the homeomorphism $X_K\to X_K^+$ maps $Y_K$ onto $Y_K^+$.

Finally, by Proposition~\ref{pgalois} the Galois group $\gal(\kab/H_+(K))$ is the image of $\hat\OO^*$ under the Artin map, so $\gal(\kab/H_+(K))\times_{\hat\OO^*}\hat\OO$ is the image of $\hat\OO^*\times\hat\OO\subset \akf^*\times\akf$ in $X_K$. It follows that the image of $\gal(\kab/H_+(K))\times_{\hat\OO^*}\hat\OO$ in~$X^+_K=J_K\times_{K^*_+/\OO^*_+}(\akf/\overline{\OO^*_+})$ is the image of $\hat\OO^*\times\hat\OO\subset \akf^*\times\akf$ under the quotient map, so it is $\{\OO\}\times\hat\OO/\overline{\OO^*_+}=i(\hat\OO/\overline{\OO^*_+})$.
\end{proof}

We can now state one of our main results.

\begin{theorem}\label{tiso}
The homeomorphism from \lemref{lxhomeo} gives rise to a canonical isomorphism of C$^*$-dynamical systems $(A_K,\sigma^K) \cong (A_K^+,\sigma^{K,+})$. This induces an isomorphism
\[
\redheck{P^+_K}{P^+_\OO} \cong p_KA_Kp_K
\]
of our Hecke algebra onto the corner of $A_K$ defined by the full projection $p_K$
corresponding to the compact open subset $Z_{H_+(K)} \subset Y_K$ from \lemref{lxhomeo}.
\end{theorem}

\bp
It follows immediately from Lemma~\ref{lxhomeo} that
the homeomorphism of $X_K$ to $X_K^+$ induces an isomorphism
$(A_K,\sigma^K)\cong(A_K^+,\sigma^{K,+})$  mapping $p_KA_Kp_K$ onto
$$
\chf_{i(\hat\OO/\overline{\OO^*_+})}A_K^+\chf_{i(\hat\OO/\overline{\OO^*_+})}
=\chf_{i(\hat\OO/\overline{\OO^*_+})}(C_0(X_K^+)\rtimes J_K)\chf_{i(\hat\OO/\overline{\OO^*_+})}.
$$
By Proposition~\ref{pnoncomind}, the latter algebra is isomorphic to $\chf_{\hat\OO/\overline{\OO^*_+}}(C_0(\akf/\overline{\OO^*_+})\rtimes(K^*_+/\OO^*_+))\chf_{\hat\OO/\overline{\OO^*_+}}$, which is in turn isomorphic to $\redheck{P^+_K}{P^+_\OO}$ by Proposition~\ref{pcross}.
The projection $p_K$ is full because $J_Ki(\hat\OO/\overline{\OO^*_+})=X_K^+$.
\ep

Therefore the Bost-Connes system for $K$ can be constructed from $\redheck{P^+_K}{P^+_\OO}$ by first dilating the semigroup crossed product decomposition of the Hecke algebra to a crossed product by the group $\mathcal P_{K,+}\cong K^*_+/\OO^*_+$ of
principal fractional ideals with a totally positive generator, then inducing from $\mathcal P_{K,+}$ to $J_K$, and finally restricting to a natural corner.

\smallskip

As an easy application we can classify KMS-states of  the Hecke C$^*$-algebra
$\redheck{P^+_K}{P^+_\OO}\cong C(\hat\OO/\overline{\OO^*_+})\rtimes(\OO^\times_+/\OO^*_+)$ with respect to the canonical dynamics. To formulate the result, for an element $c$ of the narrow class group $\cl_+(K)$ denote by $\zeta(\cdot,c)$ the corresponding partial zeta function,
$$
\zeta(s,c)=\sum_{\aaa\in J_K^+\colon \aaa\in c}N_K(\aaa)^{-s}.
$$

\begin{theorem} \label{theckekms}
For the system $(C(\hat\OO/\overline{\OO^*_+})\rtimes(\OO^\times_+/\OO^*_+),\sigma)$ we have:
\enu{i} for every $\beta\in(0,1]$ there is a unique KMS$_\beta$-state, and it is of type III$_1$;
\enu{ii} for every $\beta\in(1,\infty)$ extremal KMS$_\beta$-states are of type I and are indexed by the subset $Y^+_{K,0}\subset X^+_K=J_K\times_{K^*_+/\OO^*_+}(\akf/\overline{\OO^*_+})$ defined by
$
Y^+_{K,0}=\{(g,\omega)\mid g\omega\in\hat\OO^*/\hat\OO^*\}$; explicitly, the state~$\varphi_{\beta,x}$ corresponding to $x=(g,\omega)\in Y^+_{K,0}$ factors through the canonical conditional expectation onto $C(\hat\OO/\overline{\OO^*_+})$, and on $C(\hat\OO/\overline{\OO^*_+})$ it is given by
$$
\varphi_{\beta,x}(f)=\frac{1}{\zeta(\beta,c_x)}\sum_{h\in (K^*_+/\OO^*_+)\cap gJ_K^+}N_K(hg^{-1})^{-\beta}f(h\omega),
$$
where $c_x\in\cl_+(K)$ is the class of $g^{-1}$.
\end{theorem}

\bp By Theorem~\ref{tiso} the system
$(C(\hat\OO/\overline{\OO^*_+})\rtimes(\OO^\times_+/\OO^*_+),\sigma)$
is isomorphic to the full corner $(p_KA_Kp_K,\sigma^K)$ of the
Bost-Connes system. By \cite[Theorem 3.2]{LN} there is a one-to-one correspondence between
KMS-weights of equivariantly Morita
equivalent algebras. In our case we deal with unital C$^*$-algebras, so every densely defined weight is finite. Therefore for every $\beta\in\R$ the map $\varphi\mapsto\varphi(p_K)^{-1}\varphi|_{p_KA_Kp_K}$ is a bijection between KMS$_\beta$-states on~$A_K$ and those on $p_KA_Kp_K$. A more elementary way to check that this is a bijection (at least for $\beta\ne0$) is to apply \cite[Proposition~1.1]{LLNlat} to reduce the study of KMS-states for both systems to a study of measures satisfying certain scaling and normalization conditions. Once we have this bijection, we just have to translate the classification of KMS-states for the Bost-Connes system to our setting.

Part (i) is an immediate consequence of \cite[Theorem 2.1]{LLNlat} and \cite[Theorem 2.1]{nes2}.

As for part (ii), by \cite[Theorem 2.1]{LLNlat} for every
$\beta\in(1,+\infty)$ extremal KMS$_\beta$-states on $A_K$ are indexed
by the set $Y_{K,0}:=\gal(\kab/K)\times_{\hat\OO^*}\hat\OO^*\subset
Y_K$: the state corresponding to $x\in Y_{K,0}$ is defined by the
probability measure $\mu_{\beta,x}$ on $Y_K$ which is concentrated on
$J_K^+x$ and has the property
$\mu_{\beta,x}(hx)=N_K(h)^{-\beta}\mu_{\beta,x}(x)$ for $h\in
J_K^+$. It is easy to see that the homeomorphism $\phi:X_K\to X^+_K$
from Lemma~\ref{lxhomeo} maps $Y_{K,0}$ onto $Y^+_{K,0}$. Thus extremal
KMS$_\beta$-states for
$(C(\hat\OO/\overline{\OO^*_+})\rtimes(\OO^\times_+/\OO^*_+),\sigma)$
are indexed by the set $Y^+_{K,0}$. The state $\varphi_{\beta,x}$
corresponding to $x\in Y^+_{K,0}$ is defined by the
measure~$\nu_{\beta,x}$ which is concentrated on $i^{-1}(J_k^+x)$,
where $i\colon\akf/\overline{\OO^*_+}\hookrightarrow X^+_K$ is the
canonical embedding, and is determined by the property that
$\nu_{\beta,x}(i^{-1}(hx))=N_K(h)^{-\beta}c$ for every $h\in J_K^+$
such that $hx\in i(\hat\OO/\overline{\OO^*_+})$, where~$c$ is a
uniquely defined normalization constant. If $(g,\omega)\in
J_K\times(\akf/\overline{\OO^*_+})$ is a representative of $x\in
Y^+_{K,0}\subset
J_K\times_{K^*_+/\overline\OO^*_+}(\akf/\overline{\OO^*_+})$ then
$hgx\in i(\hat\OO/\overline{\OO^*_+})$ for $h\in J_K^+$ if and only if $hg\in K^*_+/\OO^*_+$, and then $i^{-1}(hx)=(hg)\omega$. Therefore $i^{-1}(J_K^+x)$ consists of points $h\omega$ with $h\in(K^*_+/\OO^*_+)\cap gJ_K^+$, so that, up to a normalization constant, the measure $\nu_{\beta,x}$ is
$$
\sum_{h\in (K^*_+/\OO^*_+)\cap gJ_K^+}N_K(hg^{-1})^{-\beta}\delta_{h\omega}.
$$
To get a probability measure we need to divide the above sum by $\zeta(\beta,c_x)$.
\ep

\begin{remark}\mbox{\ }
\enu{i} We can equivalently say that extremal KMS$_\beta$-states for $\beta>1$ are in a one-to-one correspondence with $K^*_+/\OO^*_+$-orbits in $\akf^*/\overline{\OO^*_+}$, that is, with the set $\akf^*/K^*_+\overline{\OO^*_+}=\akf^*/\overline{K^*_+}\cong\gal(\kab/K)$. Any such orbit carries a measure $\nu$, unique up to a scalar, such that $\nu(h\omega)=N_K(h)^{-\beta}\nu(\omega)$ if $h\in K^*_+$ and~$\omega$ lies on the orbit. With a suitable normalization the part of the orbit lying in $\hat\OO/\overline{\OO^*_+}$ defines a probability measure on~$\hat\OO/\overline{\OO^*_+}$ which gives the required state. The corresponding partition function is the partial zeta function defined by the class of the orbit in $\akf^*/\ohs K^*_+\cong\cl_+(K)$.
\enu{ii} Even if the classification of KMS-states for $(A_K,\sigma^K)$
were not known, it would still be convenient to induce from
$K^*_+/\OO^*_+$ to $J_K$ and work with $A_K$ instead of
$\redheck{P^+_K}{P^+_\OO}$.  Indeed, the action of~$K^*_+/\OO^*_+$
on~$\akf/\overline{\OO^*_+}$ is more complicated than that of $J_K$ on
$X_K$, e.g.~because $K^*_+/\OO^*_+$-orbits not passing through
$\hat\OO^*/\overline{\OO^*_+}$ do not have canonical representatives,
and one would be forced to consider the set of ideals of minimal norm
in their narrow class, analogously to \cite{LFr2}.  By
contrast, $J_K$-orbits in $X_K$ enter $Y_K$
at a unique point in $Y_{K,0}$. Furthermore, the group
$\gal(\kab/K)\cong\akf^*/\overline{K^*_+}$ acts on~$A_K$ and induces a
free transitive action on extremal KMS$_\beta$-states ($\beta>1$).
Only when restricted to
$\gal(\kab/H_+(K))\cong\hat\OO^*/\overline{\OO^*_+}$ does this
action come from automorphisms of the algebra
$\redheck{P^+_K}{P^+_\OO}$. The main reason why $A_K$ is easier to
study than~$\redheck{P^+_K}{P^+_\OO}$ is that the ordered group
$(J_K,J_K^+)$ is lattice-ordered,
unlike~$(K^*_+/\OO^*_+,\OO^\times_+/\OO^*_+)$ (an intersection of two
principal ideals need not be principal).
\enu{iii}  The induced space $X_K=\gal(\kab/K)\times_{\hat\OO^*}\akf$
comes with a natural action of $\gal(\kab/K)$, which in turn induces a
symmetry of the system defined by automorphisms of the algebra $A_K$,
and not just of the KMS$_\beta$-states.
This is different from the symmetry considered in \cite{cmr1},
which comes from the action of the semigroup $\hat\OO\cap\akf^*$ on $A_K$ by endomorphisms defined by the action of $\akf^*$ on the second coordinate of $X_K=\gal(\kab/K)\times_{\hat\OO^*}\akf$. The endomorphisms defined by elements of $\hat\OO\cap\overline{K^*_+}$ are inner, so one gets a well-defined action of $(\hat\OO\cap\akf^*)/(\hat\OO\cap\overline{K^*_+})\subset\gal(\kab/K)$ on KMS$_\beta$-states, which then extends to an action of the whole Galois group $\gal(\kab/K)$.

Despite the fact that  the  two actions of $\hat\OO\cap\akf^*$ differ significantly at the C$^*$-algebra level,
they actually coincide on KMS$_\beta$-states.
The reason is that they define the same actions on the space of $J_K$-orbits of points in $Y_K^0$.
\end{remark}

\bigskip

\section{Comparison with other Hecke systems}\label{srelation}
The C$^*$-algebra associated with the Hecke inclusion of full affine groups
$$
P_\OO:=\matr{\OO}{\OO^*}\subset P_K:=\matr{K}{K^*}
$$
was studied in \cite{LFr} and \cite{LFr2}. By \cite[Theorem 2.5]{LFr} the corresponding Hecke C$^*$-algebra $\redheck{P_K}{P_\OO}$ is isomorphic to a crossed product by the semigroup of principal ideals,
$$\chf_{\hat\OO/\overline{\OO^*}}(C_0(\akf/\overline{\OO^*})
\rtimes(K^*/\OO^*))\chf_{\hat\OO/\overline{\OO^*}}=C(\hat\OO/\overline{\OO^*})\rtimes(\OO^\times/\OO^*).$$

It is known that for imaginary quadratic fields of any class number
 these Hecke systems are Morita equivalent
to Bost-Connes systems \cite[Proposition 4.6]{cmr2}. We also know from
\cite[Remark  2.2(iii)]{LLNlat} that for totally imaginary fields $K$
of class number one the Hecke systems are actually isomorphic to the
Bost-Connes systems.  In this section we will generalize these results and show that for arbitrary number fields $\redheck{P_K}{P_\OO}$ embeds into the corner of $A_K$ corresponding to the Hilbert class field.

Our construction of the corner $p_KA_Kp_K$ works for any intermediate field $L$ between $K$ and its narrow Hilbert class field~$H_+(K)$. Namely, let $\tilde r_K\colon\akf^*\to\gal(\kab/K)$ be the restriction of the Artin
map to the finite ideles.  For $K\subset L\subset H_+(K)$, put
$U_L=\tilde r_K^{-1}(\gal(\kab/L))$.  We have $\akf^*=U_K\supset U_L\supset
U_{H_+(K)}=K^*_+\hat\OO^*$.  For example, when $L = H(K)$ is the Hilbert class field, we have $U_{H(K)}=K^*\hat\OO^*$. These descriptions of $U_K$, $U_{H(K)}$, and $U_{H_+(K)}$ are the content of  Proposition~\ref{pgalois}.

Put
$I_L=U_L/\ohs\subset J_K$. The action $g(x,y)=(xg^{-1},gy)$ of $U_L$
on $U_L\times\akf$ descends to an action of $I_L$ on $(U_L/\overline{K^*_+})\times_\ohs\akf\cong\gal(\kab/L)\times_\ohs\akf$. Then similarly to Theorem~\ref{tiso} we have the following result.

\begin{theorem}\label{tinter}
The map
$
\akf^*\times\akf\to\akf^*\times U_L\times\akf,
$
defined by $(x,y)\mapsto(x^{-1},1,xy),$ induces a $J_K$-equivariant homeomorphism
$$
\gal(\kab/K)\times_\ohs\akf\cong J_K\times_{I_L}(\gal(\kab/L)\times_\ohs\akf).
$$
This homeomorphism in turn induces an isomorphism of C$^*$-algebras
$$
q_LA_Kq_L\cong C(\gal(\kab/L)\times_\ohs\hat\OO)\rtimes I_L^+,
$$
where $q_L=\chf_{Z_L}$ is the projection  corresponding to the subset
$Z_L=\gal(\kab/L)\times_\ohs\hat\OO\subset Y_K$, and $I_L^+=I_L\cap J_K^+$ is the subsemigroup of integral ideals in $I_L$.
\end{theorem}

\begin{remark}
Recall from \cite{cmr1,LLNlat} that $A_K$ can be interpreted as the algebra of the equivalence relation of commensurability of $1$-dimensional $K$-lattices divided by (the closure of) the scaling action of~$K^o_\infty$. Then the subalgebra $q_LA_Kq_L$ corresponds to lattices that are up to scaling defined by ideals in~$I_L$. For $L=H_+(K)$ the algebra $q_LA_Kq_L$ has an interpretation as a Hecke algebra, and hence a presentation
derived from the multiplication table of double cosets. It would be interesting to see whether $q_LA_Kq_L$ has a similar natural presentation for other $L$.
\end{remark}

The relation between the Hecke algebra $\redheck{P_K}{P_\OO}$ from \cite{LFr} and the Bost-Connes algebra $A_K$ is obtained by setting $L$ to be the Hilbert class field.  The result generalizes Remark 33(b) in \cite{bos-con}, made for~$K=\Q$.

\begin{proposition} \label{phecke} We have
$
q_{H(K)}A_Kq_{H(K)}\cong C(\gal(\kab/H(K))\times_\ohs\hat\OO)\rtimes (\OO^\times/\OO^*)$ and
$$
q_{H(K)}A_K^{r_K(K^*_\infty)}q_{H(K)}=(q_{H(K)}A_Kq_{H(K)})^{r_K(K^*_\infty)}\cong \redheck{P_K}{P_\OO}.
$$
\end{proposition}

Note that $r_K(K^*_\infty)$ is a finite group of order not bigger than $2^r$, where $r$ is the number of real embeddings of~$K$.

\bp[Proof of Proposition~\ref{phecke}]
The first isomorphism is just \thmref{tinter} with $L = H(K)$.
Since $r_K(K^*_\infty)\subset\gal(\kab/H(K))$, the projection $q_{H(K)}$ is $r_K(K^*_\infty)$-invariant, so $$q_{H(K)}A_K^{r_K(K^*_\infty)}q_{H(K)}=(q_{H(K)}A_Kq_{H(K)})^{r_K(K^*_\infty)}.$$
As was observed in the proof of Proposition~\ref{pgalois}, we have $r_K(K^*_\infty)=\tilde r_K(K^*)$. Therefore, using that $\gal(\kab/H(K))\cong K^*\ohs/\overline{K^*_+}$, we get
$$
\gal(\kab/H(K))/r_K(K^*_\infty)\cong K^*\ohs/K^*\overline{K^*_+}=K^*\ohs/\overline{K^*}\cong\ohs/\overline{\OO^*}.
$$
As $(\ohs/\overline{\OO^*})\times_\ohs\akf\cong\akf/\overline{\OO^*}$, we thus have an $I_{H(K)}$-equivariant homeomorphism between the quotient of $\gal(\kab/H(K))\times_\ohs\akf$ by the action of $r_K(K^*_\infty)$ and the space~$\akf/\overline{\OO^*}$, so that
$$
(C(\gal(\kab/H(K))\times_\ohs\hat\OO)\rtimes (\OO^\times/\OO^*))^{r_K(K^*_\infty)}\cong
C(\hat\OO/\overline{\OO^*})\rtimes (\OO^\times/\OO^*).
$$
Since the latter algebra is isomorphic to $\redheck{P_K}{P_\OO}$ by \cite[Theorem 2.5]{LFr} (see also \cite[Definition 2.2]{LFr}), we conclude that $(q_{H(K)}A_Kq_{H(K)})^{r_K(K^*_\infty)}\cong \redheck{P_K}{P_\OO}.$
\ep

\begin{remark}\mbox{\ }
\enu{i} Since $\gal(H_+(K)/K)\cong \akf^*/K^*_+\ohs\cong \cl_+(K)$ and $\gal(H(K)/K)\cong\akf^*/K^*\ohs\cong\cl(K)$, the fields $H_+(K)$ and $H(K)$ coincide if and only if $K^*_+/\OO^*_+=K^*/\OO^*$, that is, $K^*=\OO^*K^*_+$. In this case the above result implies that $\redheck{P_K}{P_\OO}$ is isomorphic to a fixed point subalgebra of~$\redheck{P_K^+}{P_\OO^+}$ under a finite group action. This is easy to see by definition of Hecke algebras: the isomorphism simply comes from the restriction map $\hecke{P_K}{P_\OO}\to\hecke{P^+_K}{P^+_\OO}$, $f\mapsto f|_{P^+_K}$, and as a finite group we can take $\OO^*/\OO^*_+$, with the action defined by conjugation by matrices $\diag{1}{x}$, $x\in\OO^*$. Observe that in this case the group $ r_K(K^*_\infty)\cong\overline{K^*}/\overline{K^*_+}\cong \overline{\OO^*}/\overline{\OO^*_+}$ is a quotient of $\OO^*/\OO^*_+$.
\enu{ii} The previous proposition can be used to apply the classification of KMS-states of the Bost-Connes system for $K$ to analyze KMS-states of $\redheck{P_K}{P_\OO}$. Namely, it follows from \cite[Proposition~1.1]{LLNlat} that for $\beta\ne0$ KMS$_\beta$-states on $\redheck{P_K}{P_\OO}$ are in a one-to-one correspondence with measures on
$$\akf/\overline{\OO^*}\cong(\ohs/\overline{\OO^*})\times_\ohs\akf
\cong(\gal(\kab/H(K))\times_\ohs\akf)/r_K(K^*_\infty)$$
satisfying certain scaling and normalization conditions. Any such  measure defines an $r_K(K^*_\infty)$-invariant measure on $\gal(\kab/H(K))\times_\ohs\akf$ satisfying similar conditions, hence it gives a KMS$_\beta$-state on the algebra $q_{H(K)}A_Kq_{H(K)}$. Thus we have a bijection between KMS$_\beta$-states on $\redheck{P_K}{P_\OO}$ and $r_K(K^*_\infty)$-invariant KMS$_\beta$-states on $q_{H(K)}A_Kq_{H(K)}$, or equivalently, on $A_K$. Using this we get a result for~$\redheck{P_K}{P_\OO}$ similar to Theorem~\ref{theckekms}, but with ``pluses erased". We leave details to the interested reader, limiting ourselves to pointing out that in this case the role of $Y^+_{K,0}$ is played
by the subset $\{(g,\omega)\mid g\omega\in\ohs/\ohs\}\cong\akf^*/\overline{K^*}\cong \gal(\kab/K)/r_K(K^*_\infty)$ of the set
$$
J_K\times_{K^*/\OO^*}(\akf/\overline{\OO^*})\cong (\akf^*/\overline{K^*})\times_{\hat\OO^*}\akf
\cong (\gal(\kab/K)\times_\ohs\akf)/r_K(K^*_\infty).
$$
In particular, for every $\beta>1$ we have a free transitive action of $\gal(\kab/K)/r_K(K^*_\infty)$ on the set of extremal KMS$_\beta$-states of $\redheck{P_K}{P_\OO}$. This completes and simplifies the analysis in~\cite{LFr2}.
\enu{iii} Another topological Hecke pair naturally associated with $K$ is
$$
\Gamma=\matr{\hat\OO}{\ohs}\subset G=\matr{\akf}{\akf^*}.
$$
The corresponding C$^*$-algebra is isomorphic to the symmetric part $A_K^{\gal(\kab/K)}$ of the Bost-Connes system for $K$. Indeed, if $p\in C_r^*(G)$ is the projection corresponding to the compact open subgroup~$\Gamma$ of $G$, then similarly to the proof of Proposition~\ref{pcross} we have
$$
\redheck{G}{\Gamma}=pC_r^*(G)p\cong\chf_{\hat\OO/\ohs}(C_0(\akf/\ohs)
\rtimes(\akf^*/\ohs))\chf_{\hat\OO/\ohs},
$$
and it remains to note that $\akf/\ohs=X_K/\gal(\kab/K)$.
\end{remark}

\bigskip

\section{Functoriality of Bost-Connes systems}\label{sfunctor}
Consider an embedding $\sigma\colon K\hookrightarrow L$ of number
fields. We also denote by $\sigma$ other embeddings which it induces,
e.g.~of $\ak\hookrightarrow\al$, $\akf^*\hookrightarrow\alf^*$,
$J_K\hookrightarrow J_L$, etc. Recall that the Bost-Connes system for~$K$ is constructed using an action of $J_K$ on $X_K=\gal(\kab/K)\times_{\hat\OO^*_K}\akf$. We induce this action to an action of $J_L$ by letting
$$
X_\sigma=J_L\times_{J_K}X_K,
$$
so $X_\sigma$ is the quotient of $J_L\times X_K$ by the action
$h(g,x)=(g\sigma(h)^{-1},hx)$ of $J_K$. We want to compare the action
of $J_L$ on $X_\sigma$ with that on $X_L$.

Consider the map $\sigma\times\sigma\colon \ak^*\times\akf\to\al^*\times\alf$. Identifying $X_K$ and $X_L$ with quotients of $\ak^*\times\akf$ and $\al^*\times\alf$, respectively, we then get a map $X_K\to X_L$, which we continue to denote by $\sigma$. Note that on the level of Galois groups it is defined using  the transfer map $V_{L/\sigma(K)}\colon \gal(\sigma(K)^{ab}/\sigma(K))\to\gal(\lab/L)$, see property~\eqref{efunct} of the Artin map in Section~\ref{sprelim}.

The map $\sigma\colon X_K\to X_L$ is $J_K$-equivariant in the sense that $\sigma(hx)=\sigma(h)\sigma(x)$ for $h\in J_K$ and $x\in X_K$. It follows that we have a well-defined map
$$
\pi_\sigma\colon X_\sigma\to X_L,\ \ \pi_\sigma(g,x)=g\sigma(x).
$$

\begin{lemma}
The map $\pi_\sigma\colon X_\sigma=J_L\times_{J_K}X_K\to X_L$ is $J_L$-equivariant and its image is dense.
\end{lemma}

\bp Equivariance is clear. To show density it is enough to show that
the $J_L$-orbit of the point $(e,1)\in
X_L=\gal(\lab/L)\times_{\ohs_L}\alf$ is dense. By Lemma~\ref{lxhomeo}
we have a $J_L$-equivariant homeomorphism $X_L\to
J_L\times_{L^*_+/\OO^*_{L,+}}(\alf/\overline{\OO^*_{L,+}})$, which
maps $(e,1)$ into $(\OO_L,1)$. Therefore density of the $J_L$-orbit of
$(e,1)$ is equivalent to density of $L^*_+$ in $\alf$, and the latter can be showed as follows. Take an arbitrary open set
in $\alf$ of the form $U=\prod_{v\in S} U_v\times\prod_{v\notin
  S}\OO_v$ for some finite set of places~$S$.  We know that $L$ is dense in
$\alf$, so we can find an element $l\in L\cap U$.  Let $p_1, \dots, p_s$ be the
integer primes below the primes in $S$.  Take an integer $N$ big
enough for the integer $n=(p_1\dots p_s)^N$ to satisfy a) $n+U=U$ and
b) $n>\iota(-l)$ for all real embeddings $\iota\colon L\hookrightarrow\R$. Then $n+l\in L_+^*\cap U$.
\ep

The map $\pi_\sigma$ is not proper unless $\sigma(K)=L$, which can be seen e.g.~from Proposition~\ref{pinduction}(ii) below. It defines a $J_L$-equivariant injective homomorphism $C_0(X_L)\to C_b(X_\sigma)$, hence an injective homomorphism
$$
\pi_\sigma^*\colon C_0(X_L)\rtimes J_L\to M(C_0(X_\sigma)\rtimes J_L).
$$
On the other hand, we have a $J_K$-equivariant embedding $i_\sigma\colon X_K\hookrightarrow X_\sigma$, $x\mapsto(\OO_L,x)$. By \proref{pnoncomind} it gives us an isomorphism
$$
i_\sigma^*\colon \chf_{i_\sigma(X_K)}(C_0(X_\sigma)\rtimes J_L)\chf_{i_\sigma(X_K)}\to C_0(X_K)\rtimes J_K.
$$
Thus we can define a $(C_0(X_L)\rtimes J_L)$-$(C_0(X_K)\rtimes J_K)$-correspondence, that is, a right Hilbert $(C_0(X_K)\rtimes X_K)$-module with a left action of $C_0(X_L)\rtimes J_L$, by
$$
\tilde A_\sigma=(C_0(X_\sigma)\rtimes J_L)\chf_{i_\sigma(X_K)},\ \ \langle \xi,\zeta\rangle=i^*_\sigma(\xi^*\zeta).
$$
The actions of $C_0(X_L)\rtimes J_L$ and $C_0(X_K)\rtimes X_K$ are given by $\pi^*_\sigma$ and $(i^*_\sigma)^{-1}$. Since $J_Li_\sigma(X_K)=X_\sigma$, the projection $\chf_{i_\sigma(X_K)}\in M(C_0(X_\sigma)\rtimes J_L)$ is full. As $\pi^*_\sigma$ is injective, it follows that the left action of $C_0(X_L)\rtimes J_L$ is faithful.

It will be convenient to have the following description of the Hilbert module $\tilde A_\sigma$. Consider $C^*(J_L)$ as a right Hilbert $C^*(J_K)$-module $C^*(J_L)_\sigma$ with the right module structure defined by the embedding $C^*(J_K)\hookrightarrow C^*(J_L)$ defined by $\sigma$, and the $C^*(J_K)$-valued inner product $\langle\xi,\zeta\rangle=\sigma^{-1}(E(\xi^*\zeta))$, where $E\colon C^*(J_L)\to C^*(\sigma(J_K))$ is the canonical conditional expectation, so $E(u_g)=0$ for $g\in J_L\setminus \sigma(J_K)$.

\begin{lemma} \label{liso}
We have a canonical isomorphism $\tilde A_\sigma\cong C^*(J_L)_\sigma\otimes_{C^*(J_K)}(C_0(X_K)\rtimes J_K)$
of right Hilbert $(C_0(X_K)\rtimes J_K)$-modules. Under this isomorphism the left action of $C_0(X_L)\rtimes J_L$ is given by
$$
u_gf(u_h\otimes\xi)=u_{gh}\otimes f(h\sigma(\cdot))\xi\ \ \hbox{for}\ \ g,h\in J_L,\ \ f\in C_0(X_L)\ \ \hbox{and}\ \ \xi\in C_0(X_K)\rtimes J_K.
$$
\end{lemma}

\bp The module $\tilde A_\sigma$ is the closed linear span of elements of the form $u_hf\in C_0(X_\sigma)\rtimes J_L$ with $\supp f\subset i_\sigma(X_K)$. It is then straightforward to check that the map $u_hf\mapsto u_h\otimes f(i_\sigma(\cdot))$ is the required isomorphism.
\ep

Recalling now that the C$^*$-algebra of the Bost-Connes system for $K$ is $A_K=C(Y_K)\rtimes J_K^+=\chf_{Y_K}(C_0(X_K)\rtimes J_K)\chf_{Y_K}$, where $Y_K=\gal(\kab/K)\times_{\ohs_K}\hat\OO_K$, we can define an $A_L$-$A_K$-correspondence~by
$$
A_\sigma=\chf_{Y_L}\tilde A_\sigma\chf_{Y_K}.
$$
Observe that since $\chf_{Y_K}$ is a full projection in $C_0(X_K)\rtimes J_K$, the left action of $C_0(X_L)\rtimes Y_L$ on $\tilde A_\sigma\chf_{Y_K}$ is still faithful. Hence the left action of $A_L$ on $A_\sigma$ is faithful.

\begin{lemma}
Assume $\sigma\colon K\to L$ and $\tau\colon L\to E$ are embeddings of number fields. Then we have a canonical isomorphism
$
A_\tau\otimes_{A_L}A_\sigma\cong A_{\tau\circ\sigma}
$
of $A_E$-$A_K$-correspondences.
\end{lemma}

\bp Using Lemma~\ref{liso} we get the following isomorphisms of right Hilbert $(C_0(X_K)\rtimes J_K)$-modules:
\begin{align*}
\tilde A_\tau\otimes_{C_0(X_L)\rtimes J_L}\tilde A_\sigma
&\cong \big(C^*(J_E)_\tau\otimes_{C^*(J_L)}(C_0(X_L)\rtimes J_L)\big)\otimes_{C_0(X_L)\rtimes J_L}\tilde A_\sigma\\
&\cong C^*(J_E)_\tau\otimes_{C^*(J_L)} \tilde A_\sigma\\
&\cong C^*(J_E)_\tau\otimes_{C^*(J_L)}\big(C^*(J_L)_\sigma\otimes_{C^*(J_K)}(C_0(X_K)\rtimes J_K)\big)\\
&\cong C^*(J_E)_{\tau\circ\sigma}\otimes_{C^*(J_K)}(C_0(X_K)\rtimes J_K)\\
&\cong \tilde A_{\tau\circ\sigma}.
\end{align*}
It is easy to see that these isomorphisms respect the left actions of $C_0(X_E)\rtimes J_E$. The lemma is now a consequence of the following general result. If $A$ and $B$ are C$^*$-algebras, $X$ is a right Hilbert $A$-module, $Y$ is an $A$-$B$-correspondence and $p\in A$ is a full projection then the map
$$
Xp\otimes_{pAp}pY\to X\otimes_AY,\ \ \xi\otimes\zeta\mapsto\xi\otimes\zeta,
$$
is an isomorphism of right Hilbert $B$-modules. Indeed, we have
$$
Xp\otimes_{pAp}pY\cong X\otimes_AAp\otimes_{pAp}pA\otimes_AY,
$$
so the result follows from the isomorphism $Ap\otimes_{pAp}pA\cong A$, $a\otimes b\mapsto ab$, of $A$-$A$-correspondences.
\ep

The correspondences we have constructed are not quite compatible with the dynamics of Bost-Connes systems, because $N_L\circ\sigma=N_K^{[L:\sigma(K)]}$. It is therefore natural to replace the absolute norm~$N_K$ by the normalized norm $\tilde N_K:=N_K^{1/[K:\Q]}$, and define a dynamics $\tilde\sigma^K$ on $A_K\subset C_0(X_K)\rtimes J_K$ by
$$
\tilde\sigma^K_t(fu_g)=\tilde N_K(g)^{it}fu_g=N_K(g)^{it/[K:\Q]}fu_g=\sigma^K_{t/[K: \Q]}(fu_g).
$$

For an embedding $\sigma\colon K\to L$ of number fields we define a one-parameter group of isometries $U^\sigma$ on $A_\sigma\subset C_0(X_\sigma)\rtimes J_L$ by
$$
U^\sigma_tfu_g=\tilde N_L(g)^{it}fu_g=N_L(g)^{it/[L:\Q]}fu_g.
$$
The correspondence $A_\sigma$ then becomes equivariant for the dynamical systems $(A_L,\tilde\sigma^L)$ and $(A_K,\tilde\sigma^K)$ in the sense that
$$
U^\sigma_ta\xi=\tilde\sigma^L_t(a)U^\sigma_t\xi \ \hbox{for} \ a\in A_L,\ \ U^\sigma_t(\xi a)=(U^\sigma_t\xi)\tilde\sigma^K_t(a) \ \hbox{for}  \ a\in A_K,\ \ \langle U^\sigma_t\xi,U^\sigma_t\zeta\rangle=\tilde\sigma^K_t(\langle\xi,\zeta\rangle).
$$

It is clear that the isomorphism $A_\tau\otimes_{A_L}A_\sigma\cong A_{\tau\circ\sigma}$ is equivariant with respect to the actions of~$\R$ by isometries $U^\tau_t\otimes U^\sigma_t$ on $A_\tau\otimes_{A_L}A_\sigma$ and $U^{\tau\circ\sigma}_t$ on $A_{\tau\circ\sigma}$.

\smallskip

Summarizing properties of the correspondences $A_\sigma$ we get the following result.

\begin{theorem}\label{tfunctor}
The maps $K\mapsto(A_K,\tilde\sigma^K)$ for number fields $K$ and $\sigma\mapsto(A_\sigma,U_\sigma)$ for embeddings $\sigma \colon K\to L$ of number fields, define a functor from the category of number fields with embeddings as morphisms into the category of C$^*$-dynamical systems with isomorphism classes of $\R$-equivariant correspondences as morphisms.
\end{theorem}

It is natural to ask whether this functor is injective on objects and morphisms. A related problem has been recently studied in~\cite{CM}, where it is shown that the systems $(A_K,\sigma^K)$ and $(A_L,\sigma^L)$ are isomorphic (via an isomorphism of a particular form) if and only if $K$ and $L$ are isomorphic.

\smallskip

Next we will check how KMS-states for Bost-Connes systems behave under induction with respect to correspondences $A_\sigma$. For this we shall use the general construction of induced KMS-weights~\cite{LN}.

Assume $A$ is a C$^*$-algebra with a one-parameter group of automorphisms $\sigma$, $X$ is a right Hilbert $A$-module, and $U$ is a one-parameter group of isometries on $X$ such that  $U_t(\xi a)=(U_t\xi)\sigma_t(a)$ and $\langle U_t\xi,U_t\zeta\rangle=\sigma_t(\langle\xi,\zeta\rangle)$ (the first condition is in fact a consequence of the second). Then $U$ defines a strictly continuous $1$-parameter group of automorphisms~$\sigma^U$ on the C$^*$-algebra $B(X)$ of adjointable operators on $X$, $\sigma^U_t(T)=U_tTU_{-t}$. Assume $\varphi$ is a $\sigma$-KMS$_\beta$ weight on $A$, so $\varphi$ is $\sigma$-invariant, lower semicontinuous, densely defined and
$\varphi(x^*x)=\varphi(\sigma_{-i\beta/2}(x)\sigma_{-i\beta/2}(x)^*)$ for every~$x$ in the domain of definition of $\sigma_{-i\beta/2}$. By \cite[Theorem~3.2]{LN} there exists a unique $\sigma^U$-KMS$_\beta$ weight $\Phi$ on the C$^*$-algebra~$K(X)$ of generalized compact operators on $X$ such that
$$
\Phi(\theta_{\xi,\xi})=\varphi(\langle U_{i\beta/2}\xi,U_{i\beta/2}\xi\rangle)
$$
for every $\xi\in X$ in the domain of definition of $U_{i\beta/2}$, where $\theta_{\xi,\xi}\in K(X)$ is the operator defined by $\theta_{\xi,\xi}\zeta=\xi\langle\xi,\zeta\rangle$. Furthermore, the weight $\Phi$ extends uniquely to a strictly lower semicontinuous weight on $B(X)$. We will denote this weight by $\Ind^U_X\varphi$.

Induced weights behave in the expected way with respect to induction in stages. Namely, assume~$B$ is another C$^*$-algebra with dynamics~$\gamma$ and $Y$ is a right Hilbert $B$-module with a one-parameter group of isometries $V$ such that $\langle V_t\xi,V_t\zeta\rangle=\gamma_t(\langle\xi,\zeta\rangle)$. Assume further that $B$ acts on the left on $X$ and $U_tb\xi=\gamma_t(b)U_t\xi$. By \cite[Proposition~3.4]{LN} if the restriction of $\Ind^U_X\varphi$ to $B$ is densely defined then
$$
\Ind^V_Y((\Ind^U_X\varphi)|_B)=\Ind^{V\otimes U}_{Y\otimes_B X}\varphi\ \ \hbox{on}\ \ B(Y).
$$

Returning to Bost-Connes systems, recall that by \cite[Proposition 1.1]{LLNlat} for every $\beta\ne0$ there is a one-to-one correspondence between positive $\sigma^K$-KMS$_\beta$-functionals on $A_K$ and measures $\mu$ on $X_K$ such that $\mu(Y_K)<\infty$ and $\mu(gZ)=N_K(g)^{-\beta}\mu(Z)$ for $g\in J_K$ and Borel subsets $Z\subset X_K$. Such a measure defines a weight on $C_0(X_K)$. By composing it with the canonical conditional expectation $C_0(X_K)\rtimes J_K\to C_0(X_K)$, we get a weight on the crossed product, and its restriction to $A_K$ gives the required functional corresponding to $\mu$. It follows from \cite[Proposition 1.2]{LLNlat} that for $\beta>1$ such a measure $\mu$ is completely determined by its restriction to $Y_{K,0}=\gal(\kab/K)\times_{\ohs_K}\hat\OO_K^*$, and any finite measure~$\nu$ on~$Y_{K,0}$ extends uniquely to a measure $\mu$ on $X_K$ satisfying the above conditions. We denote the corresponding functional on $A_K$ by $\varphi_{\beta,\nu}$. Then $\varphi_{\beta,\nu}(1)=\zeta_K(\beta)\nu(Y_{K,0})$, where $\zeta_K$ is the Dedekind zeta function. One the other hand, for every $\beta\in(0,1]$ there is a unique KMS$_\beta$-state, and for the corresponding measure~$\mu$ we have $\mu(Y_{K,0})=0$, see \cite[Theorem~2.1]{LLNlat}.

\begin{proposition} \label{pinduction}
Let $L/K$ be an extension of number fields with $K\ne L$, $\varphi$ a $\sigma^K$-KMS$_{[L: K]\beta}$-state (hence a $\tilde\sigma^K$-KMS$_{[L:\Q]\beta}$-state) on $A_K$. Put $\Phi=(\Ind^{U_\sigma}_{A_\sigma}\varphi)|_{A_L}$, where $\sigma\colon K\to L$ is the identity map, so $\Phi$ is a weight satisfying the $\sigma^L$-KMS$_\beta$-condition but possibly not densely defined. Then
\enu{i} if $\beta>1$ and $\varphi=\varphi_{[L: K]\beta,\nu}$ for a measure $\nu$ on $Y_{K,0}$ then $\Phi=\varphi_{\beta,\sigma_*(\nu)}$; in particular,
$$
\Phi(1)=\frac{\zeta_L(\beta)}{\zeta_K([L:K]\beta)};
$$
\enuu{ii} if $\beta\in(0,1]$ then $\Phi(1)=+\infty$.
\end{proposition}

\bp Observe first that if $p$ is a full projection in a C$^*$-algebra $A$, then induction of KMS-weights
by the $A$-$pAp$ correspondence $Ap$ simply means extension. In view of this the induction procedure for Bost-Connes systems can be described as follows. Assume $\varphi$ is defined by a measure $\mu$ on $X_K$ as described above. It defines a measure on $i_\sigma(X_K)$. This measure extends uniquely to a measure $\lambda$ on $X_\sigma$ such that
$$
\lambda(gZ)=\tilde N_L(g)^{-[L: \Q]\beta}\lambda(Z)=N_L(g)^{-\beta}\lambda(Z) \ \ \hbox{for} \ g\in J_L \ \hbox{and Borel} \ Z\subset X_\sigma.
$$
Then $\Phi$ is the weight defined by the measure $\mu_\sigma:=\pi_{\sigma*}(\lambda)$ on $X_L$. Therefore the claims are that (i) if $\beta>1$ and $\nu=\mu|_{Y_{K,0}}$ then $\mu_\sigma|_{Y_{L,0}}=\sigma_*(\nu)$, and (ii) if $\beta\in(0,1]$ then $\mu_\sigma(Y_{L})=+\infty$.

\smallskip

Assume $\beta>1$ and let $\nu=\mu|_{Y_{K,0}}$. Since the sets $gY_{K,0}$, $g\in J_K$, are pairwise disjoint and the measure $\mu$ is determined by $\nu$, we have
$$
\mu(Z)=\sum_{g\in J_K}\tilde N_K(g)^{[L: \Q]\beta}\nu(gZ\cap Y_{K,0})\ \ \hbox{for Borel}\ Z\subset X_K.
$$
In particular, $\mu$ is concentrated on $J_KY_{K,0}$. Since the sets $gi_{\sigma}(Y_{K,0})$, $g\in J_L$, are pairwise disjoint, we have a similar formula for $\lambda$, so that $\lambda$ is concentrated on $J_Li_\sigma(Y_{K,0})$. Since $\pi_\sigma(i_\sigma(Y_{K,0}))\subset Y_{L,0}$, we conclude that $\mu_\sigma$ is concentrated on $J_L Y_{L,0}$ and
$\mu_\sigma|_{Y_{L,0}}=(\pi_\sigma\circ i_\sigma)_*(\nu)=\sigma_*(\nu)$.

\smallskip

Assume now that $\beta\in(0,1]$. For $\beta>1/[L: K]$ it is immediate that $\mu_\sigma(Y_L)=+\infty$, since on the one hand $\mu_\sigma(Y_{L,0})\ge\mu(Y_{K,0})>0$, and on the other we know that if $\mu_\sigma(Y_L)<\infty$ then $\mu_\sigma(Y_{L,0})=0$. But for $\beta\le1/[L:K]$ we need a different argument.

Let $v$ be a finite place of $K$. Consider the subset $W_v$ of $Y_K=\gal(\kab/K)\times_{\ohs_K}\hat\OO_K$ which is the image of $\gal(\kab/K)\times\OO^*_{K,v}\times\prod_{w\ne v, w\nmid\infty}\OO_{K,w}$ under the quotient map. The scaling condition for~$\mu$ implies (see \cite{LLNlat}) that
$$
\mu(W_v)=1-\tilde N_K(\pp_v)^{-[L:\Q]\beta}=1-N_K(\pp_v)^{-[L: K]\beta}.
$$
Denote by $J_{L,v}^+$ the unital subsemigroup of $J_L^+$ generated by ideals $\pp_w$ with $w|v$. Then  for $g\in J^+_{L,v}$ the sets $\pi_\sigma(gi_\sigma(W_v))=g\sigma(W_v)$ are mutually disjoint and contained in $Y_L$. Hence
$$
\mu_\sigma(Y_L)\ge \sum_{g\in J^+_{L,v}}\lambda(gi_\sigma(W_v))
=\sum_{g\in J^+_{L,v}}N_L(g)^{-\beta}\mu(W_v)=\frac{1-N_K(\pp_v)^{-[L: K]\beta}}{
\prod_{w|v}(1-N_L(\pp_w)^{-\beta})}.
$$
A similar computation for a finite set $F$ of places $v\nmid\infty$  yields
$$
\mu_\sigma(Y_L)\ge\prod_{v\in F}\frac{1-N_K(\pp_v)^{-[L: K]\beta}}{
\prod_{w|v}(1-N_L(\pp_w)^{-\beta})}.
$$
We claim that for $\beta\in(0,1]$ the above expression tends to infinity as $F$ ranges over all such sets. This is obviously the case for $\beta=1$, since the denominator converges to $\zeta_L(1)^{-1}=0$, while the numerator converges to $\zeta_K([L: K])^{-1}\ne0$ (as $[L: K]\ge2$ by assumption). Therefore it suffices to check that each factor in the above product is a non-increasing function in $\beta$ on $(0,1]$. To see this write $\pp_v\OO_L$ as $\prod_{w|v}\pp_w^{s_w}$, then $N_K(\pp_v)^{[L: K]}=\prod_{w|v}N_L(\pp_w)^{s_w}$. Therefore is suffices to check that for numbers $x_1,\dots,x_n>1$ and $s_1,\dots,s_n\ge1$ the function
$$
\frac{1-x_1^{-s_1\beta}\dots x_n^{-s_n\beta}}{(1-x_1^{-\beta})\dots(1-x_n^{-\beta})}
$$
is non-increasing in $\beta$ on $(0,1]$. This in turn is easy to see using that the function $\displaystyle\frac{1-ax^{-s\beta}}{1-x^{-\beta}}$ is non-increasing for any $x>1$, $s\ge1$ and $0\le a\le1$. Thus $\mu_\sigma(Y_L)=+\infty$.
\ep

\end{document}